\documentclass{article}
\usepackage[top=1.2in,bottom=1.2in,left=1.2in,right=1.2in,marginpar=1in]{geometry}

\usepackage{graphicx} 

\usepackage{latexsym,amsfonts,amssymb,amsmath,amsthm,mathrsfs}
\usepackage{pinlabel}
\usepackage[hidelinks]{hyperref}
\usepackage{tikz}
\usepackage{todonotes}
\usepackage{comment}
\usepackage{cleveref}

\newcommand{\nc}{\newcommand}
\nc{\dmo}{\DeclareMathOperator}
\nc{\nt}{\newtheorem}

\nt{thm}{Theorem}[section]
\nt{lem}[thm]{Lemma}
\nt{prop}[thm]{Proposition}
\nt{quest}[thm]{Question}
\nt{cor}[thm]{Corollary}
\nt{fact}[thm]{Fact}
\nt{prob}[thm]{Problem}
\nt{conj}[thm]{Conjecture}

\theoremstyle{remark}
\nt{remark}[thm]{Remark}
\nt{claim}{Claim}[thm]

\theoremstyle{definition}
\nt{defn}[thm]{Definition}

\makeatletter
\newtheorem*{rep@theorem}{\rep@title}
\newcommand{\newreptheorem}[2]{%
\newenvironment{rep#1}[1]{%
 \def\rep@title{#2 \ref{##1}}%
 \begin{rep@theorem}}%
 {\end{rep@theorem}}}
\makeatother

\usepackage{thmtools, thm-restate}

\newreptheorem{theorem}{Theorem}

\newenvironment{claimproof}[1]{\par\noindent\emph{Proof:}\space#1}{\hfill $\blacksquare$}

\nc{\Z}{\mathbb{Z}} 
\nc{\R}{\mathbb{R}} 
\nc{\Q}{\mathbb{Q}} 
\nc{\N}{\mathbb{N}} 
\nc{\set}[2]{\{#1 \,|\, #2\} }

\nc{\hg}{\mathcal{H}_g}
\nc{\dg}{\mathcal{D}_g}
\nc{\bddy}{\partial}
\nc{\redsys}{\mathcal{R}}
\nc{\vbddy}{\partial_v}
\nc{\hbddy}{\partial_h}
\nc{\qieq}[1]{\stackrel{#1}{\asymp}} 
\nc{\wit}{\mathbf{W}}
\nc{\lwit}{\mathcal{L}_{\mathbf{W}}}
\nc{\antiwit}{\mathcal{K}_{\mathcal{G}}}
\nc{\mcg}{\operatorname{Mod}(S)}
\nc{\pmcg}{\operatorname{PMod}(S)}
\nc{\g}{\mathcal{G}(S)}
\nc{\crs}{\mathcal{R}}
\nc{\pants}{\mathcal{P}}
\nc{\proj}{\mathcal{L}_{\mathcal{AG}}}
\nc{\ag}{\mathcal{AG}(S)}
\nc{\da}{\mathcal{D}_g'}
\nc{\annuli}{\mathcal{A}}
\nc{\defeq}{\vcentcolon=}

\nc{\textofthreeone}{\textit{Let $\g$ be a multiarc and curve graph with a witness-based distance formula, and suppose a finitely generated group $G\leq \mcg$ acts on $\g$ by isometries. If some pure power $f$ of $h\in G$ restricts to a pseudo-Anosov on a witness $W$ for $\g$, then $h$ acts loxodromically on $\g$.}}

\title{When loxodromics are pseudo-Anosovs on witnesses}
\author{Marissa Chesser}
\date{}

\begin{document}

\maketitle

\begin{abstract}
    In this paper, we prove that for subgroups acting on admissible multiarc and curve graphs and for the handlebody group acting on the disk graph, the loxodromic elements are exactly those for which some pure power is a pseudo-Anosov on a witness. This generalizes the result of Masur and Minsky that the elements of the mapping class group that act loxodromically on the curve graph are the pseudo-Anosov elements.
\end{abstract}

\section{Introduction}\label{sec:intro}

In their notable series of papers on mapping class groups and curve graphs, Masur and Minsky include the following characterization of the elements of the mapping class group acting loxodromically on the curve graph.
\begin{thm}[{{\cite{masur-minsky-I}}}]\label{thm:masur-minsky-original}
    An element of the mapping class group acts loxodromically on the curve graph if and only if it is a pseudo-Anosov element.
\end{thm}

Over the years, other multiarc and curve graphs have been constructed in the process of understanding mapping class groups, and subgroups such as handlebody groups. For some of these graphs, there are mapping classes that act loxodromically but are not pseudo-Anosov. One might then wonder, if there is a way to classify the loxodromics in these other scenarios in a way that is analogous to the characterization of loxodromics of the mapping class for the curve graph. In this paper, we prove such an analogous characterization for the class of \textit{admissible multiarc and curve graphs} (see Definition \ref{def:admissible}) and for \emph{disk graphs}.

The charaterizations provided in this paper use the concept of \emph{witnesses}, introduced by Masur and Schleimer in \cite{MasSchleim} as \emph{holes}. A witness for a graph, whose vertices represent isotopy classes of curves associated to some surface $S$, is an essential, non-pants subsurface of $S$ that must intersect every representative of every isotopy class of curves represented by the vertices of the graph. Importantly, the only witness for the curve graph of a surface is the entire surface. Therefore, Theorem \ref{thm:masur-minsky-original} can be rephrased as follows.

\begin{reptheorem}{thm:masur-minsky-original}[{{\cite{masur-minsky-I}}}]
    An element of the mapping class group acts loxodromically on the curve graph if and only if it is pseudo-Anosov on a witness.
\end{reptheorem}

With this new phrasing, the main theorems of this paper can be seen as direct analogues of this classical result. In particular, in both the setting of admissible multiarc and curve graphs and disk graphs, the elements acting loxodromocally on these graphs are exactly those for which pure power restricts to a pseudo-Anosov on a witness.

In the context of admissible multiarc and curve graphs, the characterization follows from Corollary \ref{cor:pa-to-lox-multi} and Theorem \ref{thm:loximpliespA-multiarccurve}.
\begin{thm}\label{thm:multiarccurve}
    Let $\g$ be a connected, admissible multiarc and curve graph and $G \leq \mcg$ a finitely generated group acting by isometries on $\g$. Then an element $f\in G$ acts loxodromically on $\g$ if and only if (a pure power of) $f$ restricts to a pseudo-Anosov on a witness for the admissible multiarc and curve graph.
\end{thm}

In context of disk graphs, we have the following characterization, coming from Corollary \ref{cor:pa-to-lox-disk}, Theorem \ref{thm:lox-implies-pa-handlebody}, and Remark \ref{rem:trivial-cases}.
\begin{thm}\label{thm:handlebody}
    For $g\geq 0$, an element $f\in H_g$ acts loxodromically on the disk graph $\dg$ if and only if (a pure power of) $\bddy f$ restricts to a pseudo-Anosov on a witness for the disk graph.
\end{thm}

\begin{remark}\label{rem:trivial-cases}
    The cases of Theorem \ref{thm:handlebody} where $g\in \{0,1\}$ are trivial. We note that $V_0$ contains no essential disks, so $\mathcal{D}_0$ is empty. Additionally, $\mathcal{H}_0\simeq 1$. The handlebody $V_1$ has only one isotopy class of essential disks and $\mathcal{H}_1\simeq \Z$, consisting only of powers of Dehn twists about this disk. Thus, when $g\in\{0,1\}$, there are no elements of $\mathcal{H}_g$ acting loxodromically on $\mathcal{D}_g$ and there are no elements of $\mathcal{H}_g$ that can be pseudo-Anosov on a witness for $V_g$. So Theorem \ref{thm:handlebody} is vacuously true in these cases.
\end{remark}

On the way to proving Theorem \ref{thm:multiarccurve} and Theorem \ref{thm:handlebody}, we prove the following theorem.

\newtheorem*{thm:paimplieslox}{Theorem \ref{thm:paimplieslox}}
\begin{thm:paimplieslox}
    \textofthreeone
\end{thm:paimplieslox}

The main tool needed to prove this theorem is a witness-based distance estimate formula (see Definition \ref{def:wit-based-dist}). This theorem likely applies to graphs other than the ones examined in the present paper. In particular, if Schleimer's conjecture \cite[Conjecture 3.10]{notescomplexofcurves} that ``well-behaved'' multiarc and curve graphs are all equipped with a witness-based distance estimate formula, then Theorem \ref{thm:paimplieslox} should hold for these ``well-behaved'' graphs.

We also note that implications in Theorems Theorem \ref{thm:multiarccurve} and Theorem \ref{thm:handlebody} that a a pure power of a loxodromic element must be pseudo-Anosov on a witness are proven separately for admissible multiarc and curve graphs and for disk graphs. The proof for multiarc and curve graphs makes use of a quasi-isometric model for the multiarc and curve graph constructed by Kopreski \cite{kopreski2023multiarc} (a generalization of the graph $\mathcal{K}_{\mathcal{G}}$ of Vokes in \cite{vokes}). This type of model does not result in a quasi-isometric graph in the case of disk graphs, so instead we approach the proof from a more topological perspective. One might wonder if there is a unifying approach that would apply to both the admissible multiarc and curve graph case and the disk graph case.

\textbf{Acknowledgements:} The author would like to thank Aaron Calderon, Jacob Russell, and Chris Leininger for helpful conversations.

\section{Preliminaries}\label{sec:preliminaries}

\subsection{Coarse geometry}\label{subsec:coarse-geom}

Given metric spaces $(X,d_X)$ and $(Y,d_Y)$, a map $f:X\to Y$ is called a \emph{$(K,C)$-quasi-isometric embedding} if there exists $K\geq 1$ and $C\geq 0$ such that for every $a,b\in X$, the following inequality holds:
\[ \frac{1}{K}d_X(a,b) - C \leq d_Y(f(a),f(b)) \leq Kd_X(a,b) +C.\]
We can promote a quasi-isometric embedding to a \emph{$(K,C)$-quasi-isometry} if $f(X)$ is also \emph{$D$-dense} in $Y$, meaning there is a $D\geq 0$ such that every point $y\in Y$ is within distance $D$ of a point in $f(X)$.

We say that a metric space $(X,d_X)$ is a \emph{$G$-space} if it comes equipped with a group action by a group $G$. Given $G$-spaces $(X,d_X)$ and $(Y,d_Y)$, a map $\phi:X\to Y$ is said to be \emph{coarsely $G$-equivariant} if there is a constant $N\geq 0$ such that for any $g\in G$ and $x\in X^{(0)}$, $d(g\cdot \phi(x), \phi(g\cdot x))\leq N$. Given a group $G$, we say that an element $g\in G$ acting on a $\delta$-hyperbolic $G$-space $X$ is \emph{loxodromic} if the orbit map $\Z \to X$, defined by $n\mapsto g^n\cdot x$ for some (any) $x\in X$, is a quasi-isometric embedding.

The following lemma will be useful in translating non-loxodromic actions between $G$-spaces. 

\begin{lem}\label{lem:finitediam}
    Suppose $X$ and $Y$ are $G$-spaces and that $\phi:X\to Y$ is a quasi-isometric embedding that is coarsely $G$-equivariant. Given any base point $x\in X$, suppose $g\in G$ such that the diameter of the set of $g$-translates of $x$ in $X$ is finite. Then the set of $g$ translates of $\phi(x)$ in $Y$ must also have finite diameter.
\end{lem}

\begin{proof}
    Let $A\geq 1$ and $B\geq 0$ be the quasi-isometric constants of $\phi$ and $N\geq 0$ the coarse $G$-equivariance constant. Further, let $D<\infty$ be the diameter of the set of $g$ translates of $x$ in $X$. Then using the quasi-isometric embedding inequality, the triangle inequality, and coarse equivariance, we find for all $n\in \Z$,
    \begin{align*}
        AD +B &\geq A d_X(x, g^n\cdot x) +B \\
        & \geq d_Y(\phi(x), \phi(g^n\cdot x)) \\
        &\geq d_Y(\phi(x), g^n\cdot\phi(x)) - d_Y( g^n\cdot\phi(x), \phi(g^n\cdot x)) \\
        &\geq d_Y(\phi(x), g^n\cdot\phi(x)) -N.
    \end{align*}
    We can then conclude that for all $n\in \Z$,
    \[ AD +B +N \geq d_Y(\phi(x), g^n\cdot\phi(x)), \]
    where here none of the constants $A$, $B$, $D$, and $N$ rely on $n$. Since every $g$ translate of $\phi(x)$ remains withing $AD+B+N$ of $\phi(x)$, it follows that the diameter of the set of $g$ translates of $\phi(x)$ in $Y$ is at most $2(AD+B+N)$.
\end{proof}

\subsection{Graphs and witnesses}\label{subsec:graphs}

In this section, we define the various graphs that will be used in this paper, as well as the notion of a witness. We assume going forward that $S$ is a compact, connected, orientable, non-pants surface possibly with boundary, and with $\chi(\Sigma) \leq -1$.

\begin{defn}\label{def:multiarc-curve}
     A \textit{multiarc and curve graph} $\g$ is a simplicial graph whose vertices correspond to collections of isotopy classes of disjoint, simple, essential arcs and/or curves in $S$. Often, we allow edges between two vertices in a multiarc and curve graph if and only if there are representatives of the vertices that can be realized disjointly; however, we also allow graphs for which adjacent vertices have representatives which can be realized with uniformly bounded geometric intersection.
\end{defn}

\begin{defn}\label{def:witness}
    \emph{Witnesses} for a multiarc and curve graph $\g$ are essential, non-pants subsurfaces $X \subseteq S$ such that every representative of any $\alpha\in\g^{(0)}$ has non-empty intersection with $X$. 
\end{defn}

    The original definition of witnesses in \cite{MasSchleim} (called \emph{holes} in that paper) assumes that the witnesses are connected, and this will be the default assumption in the present paper. However, in some cases discussed, this requirement is dropped, but we will make note when we allow for disconnected witnesses.

One class of multiarc and curve graphs that we will explore in this paper are \textit{admissible} multiarc and curve graphs, as defined by Kopreski \cite{kopreski2023multiarc}, which are generalizations of the class of twist-free multicurve graphs defined by Vokes \cite[Definition 2.1]{vokes}. We note that Kopreski requires disconnected witnesses (see footnote 2 in \cite{kopreski2023multiarc}), though Vokes assumes that witnesses are connected. In this paper we will use the more general structures defined by Kopreski, but the same arguments hold for the class of twist-free multicurve graphs.

\begin{defn}[{{\cite[Definition 1.2]{kopreski2023multiarc}}}]\label{def:admissible}
    A connected multiarc and curve graph $\g$ is \emph{admissible} if
    \begin{enumerate}
        \item[(i)] $\pmcg$ preserves the set of vertices of $\g$ and extends to an action on $\g$, and
        \item[(ii)] if there is an edge between $\alpha,\beta\in\g^{(0)}$, then the geometric intersection $i(\alpha,\beta)$ is uniformly bounded.
    \end{enumerate}
\end{defn}

The other graphs that we will explore in this paper are disk graphs $\dg$ for genus $g$, orientable handlebodys $V_g$.
\begin{defn}\label{defn:disk-graph}
    A properly embedded disk $(D,\bddy D)\subseteq (V_g, \bddy V_g)$ is \emph{essential} if $\bddy D$ is essential in $\bddy V_g$. We define the vertices of the \emph{disk graph} $\dg$ to be the essential disks in $V_g$, up to ambient isotopy preserving $\bddy V_g$. There is an edge between two vertices if there are representatives of the corresponding isotopy classes of disks that are disjoint. 
\end{defn}

\section{Pseudo-Anosov on a witness implies loxodromic}\label{sec:pa-implies-lox}

The proof that a mapping class that is a pseudo-Anosov on a witness implies that it is loxodromic on the appropriate graph requires the use of distance estimate formulas.

\subsection{Distance formulas}\label{subsec:distance-formulas}
In \cite{masur-minsky}, the authors introduce the distance formula for the mapping class group, (or more specifically the marking graph $\mathcal{M}$, which is quasi-isometric to the Cayley graph of the mapping class group). The distance formula as introduced by Masur and Minsky, as well as subsequent distance formulas introduced by others, rely on the tool of subsurface projections to curve graphs. We refer the reader to \cite{masur-minsky} for the details on subsurface projections, but remind the reader that for (multi)-arcs or (multi)-curves $\alpha,\beta$ on a surface $S$, and a subsurface $X\subseteq S$, $\pi_X(\cdot)$ represents the subsurface projection to the curve graph $\mathcal CX$, and the \emph{subsurface projection distance} is defined as follows:
\[ d_X(\alpha, \beta) = \operatorname{diam}(\pi_X(\alpha)\cup \pi_X(\beta)).\]

We note also that when $X$ is allowed to be a disconnected surface, say $X = Y_1 \sqcup Y_2$, we define its curve graph $\mathcal X$ to be the graph join $\mathcal C Y_1 * \mathcal C Y_2$.

To state the Masur-Minsky distance formula, we introduce some additional notation. Given non-negative real numbers $A,B,C$ with $A \geq 1$, we write $B \qieq{A} C$ to mean $\frac{1}{A}B - A \leq C \leq AB+A$, and further define $[B]_A = B$ if $B \geq A$, and $[B]_A = 0$ otherwise. With this notation, the Masur-Minsky distance formula for a surface $S$ can be stated as follows.

\begin{thm}[{{\cite[Theorem 6.12]{masur-minsky}}}]\label{thm:distance-mm}
    There is a constant $C$, depending on $S$, such that for any $c\geq C$, there is a constant $A\geq 1$ such that for any markings $\mu, \nu \in \mathcal{M}^{(0)}$,
    \[ d_{\mathcal{M}}(\mu, \nu) \qieq{A} \sum_{X\in \mathbf{X}} [d_{X}(\mu, \nu)]_c\]
    where $\mathbf{X}$ is the set of all essential subsurfaces of $S$.
\end{thm}

There have subsequently been other examples of spaces that come equipped with Masur-Minsky type distance formulas, perhaps most notably hierarchically hyperbolic spaces introduced in \cite{HHSI,HHSII}.

We will focus on multiarc and curve graphs equipped with witness-based distance formulas, as defined below.

\begin{defn}\label{def:wit-based-dist}
    Let $\mathbf{W}$ denotes the set of witnesses for $\g$, a multiarc and curve graph. If there exists some constant $C$ such that for any $c\geq C$, there is a constant $A\geq 1$ such that for all $\alpha,\beta \in \g^{(0)}$,
    \[
        d_{\g}(\alpha,\beta) \qieq{A} \sum_{X\in \mathbf{W}} [d_X(\alpha,\beta)]_c,
    \]
    then we will say that $\g$ is equipped with a \emph{witness-based distance formula}.
\end{defn}

Definition \ref{thm:distance-mm} provides our first example of a multiarc and curve graph equipped with a witness-based distance formula, as all essential, non-pants subsurfaces of a surfaces $S$ are witnesses for the associated marking graph $\mathcal{M}$.

Admissible multiarc and curve graph are also equipped with a witness-based distance formula. This result follows from \cite[Theorem 1.3]{kopreski2023multiarc}, which shows that an admissible multiarc and curve graph is hierarchically hyperbolic with respect to subsurface projections to witnesses, and the distance formula for hierarchically hyperbolic spaces \cite[Theorem 4.5]{HHSII}.

\begin{thm}[{{\cite{kopreski2023multiarc,HHSII}}}]\label{thm:admissible-dist-formula}
    A connected, admissible multiarc and curve graph comes equipped with a witness-based distance formula, assuming disconnected witnesses.
\end{thm}

In addition, Masur and Schleimer show that the disk graph $\dg$ is equipped with a witness-based distance formula.

\begin{thm}[{{\cite[Theorem 19.9]{MasSchleim}}}]\label{thm:disk-dist-formula}
    The disk graph $\dg$ for a handlebody $V_g$ comes equipped with a witness-based distance formula.
\end{thm}

\subsection{The proof that pseudo-Anosov on a witness implies loxodromic}
For this section, assume that $G$ is some finitely generated subgroup of the mapping class group $\mcg$. In addition, assume that $\g$ is a multiarc and curve graph equipped with a witness-based distance formula (as described in  Section \ref{subsec:distance-formulas}). Assume further that $G\leq \mcg$ acts by isometries on $\g$. We prove here that for such pairs $(G,\g)$, if an element $h\in G$ restricts to a pseudo-Anosov on some witness for $\g$ (up to a pure power of $h$), then $h$ must act loxodromically on $\g$. This proof is very similar to \cite[Corollary 4.8]{russell2019nonrelative} and \cite[Lemma 7.2]{chesser22}.

\begin{thm}\label{thm:paimplieslox}
    \textofthreeone
\end{thm}

\begin{proof}
    The upper bound of the quasi-isometric embedding inequality follows from the fact that orbit maps of finitely generated groups are Lipschitz, (see for instance \cite[Chapter I.8, Lemma 8.18]{BH}).
    
    For the lower bound we use our assumption that $f|_W$ is a pseudo-Anosov. This assumption implies that for any $\beta\in \g^{(0)}$, there is some $K\geq 1$ such that
    \begin{equation}\label{eq:pAinequality}
        \frac{1}{K}|m| - K \leq d_W(\beta,f^m\cdot \beta)
    \end{equation}
    for each $m$. For a large enough $N$, we can therefore conclude that for every $|n|\geq N$, $d_W(\beta,f^n\cdot \beta)$ is greater than the $C$ given in the witness-based distance formula for $\g$ (Definition \ref{def:wit-based-dist}).
    Because $W$ is itself a witness for $\g$ and $C \leq d_W(\beta,f^n\cdot \beta)$ for each $n\geq N$, it follows that
    \begin{equation}\label{eq:suminequality}
        d_W(\beta,f^n\cdot \beta) \leq  \sum_{X\in \mathbf{W}} [d_X(\beta,f^n\cdot \beta)]_C,
    \end{equation}
    where $\mathbf{W}$ is the set of witnesses for $\g$.

    Taking $A\geq 1$ to be the constant as in the witness-based distance formula with $C=c$ as above, we can combine  Equations \ref{eq:pAinequality} and \ref{eq:suminequality} with the distance formula to find:
    \[ \frac{1}{K}|n| - K \leq d_W(\beta,f^n\cdot \beta) \leq \sum_{X\in \mathbf{W}} [d_X(\beta,f^n\cdot \beta)]_c \leq A\cdot d_{\g}(\beta, f^n\cdot \beta)+A  \]
    for any $|n|\geq N$. Hence, the upper bound is satisfied for any $|n|\geq N$ via
    \[ \frac{1}{AK}|n| - \frac{(K+A)}{A} \leq d_{\g}(\beta, f^n\cdot \beta). \]
    Then simply increase the constants until the upper bound is satisfied for the finitely many $m$ such that $0\leq |m| <N$.
    
    We thus conclude that $f$ acts loxodromically on $\g$, and since $f$ is a pure power of $h$, so too must $h$.
    \end{proof}

    \begin{remark}\label{rem:connected}
        Note that the preceding proof works whether the witness-based distance formula is assumed to use connected witnesses or allows for disconnected witnesses. Similarly, $W$ could be assumed to be connected or disconnected.
    \end{remark}

    Theorem \ref{thm:paimplieslox} can be applied to admissible multiarc and curve graphs.
    \begin{cor}\label{cor:pa-to-lox-multi}
        If $\g$ is a connected, admissible multiarc and curve graph, and $h\in \mcg$ such that some pure power restricts to a pseudo-Anosov on a witness for $\g$, then $h$ acts loxodromically on $\g$.
    \end{cor}

    \begin{proof}
        By the definition of an admissible multiarc and curve graph, we know that $\pmcg$ preserves the vertices of $\g$. We also know that the action of a mapping class must preserve geometric intersection number, so $\pmcg$ must act by isometries on $\g$.  Theorem \ref{thm:admissible-dist-formula} provides a witness-based distance formula. Therefore, by our assumption on $h$ and  Theorem \ref{thm:paimplieslox}, $h$ must act loxodromically on $\g$.
    \end{proof}

    \begin{remark}
        We note in relation to Corollary Corollary \ref{cor:pa-to-lox-multi}: though Kopreski requires that disconnected witnesses be allowed in order to define the hierarchically hyperbolic structure on the admissible multiarc and curve graphs (see footnote 2 in \cite{kopreski2023multiarc}), as noted in Remark \ref{rem:connected} , there is nothing in the proof of Theorem \ref{thm:paimplieslox} that requires the witness $W$ to be disconnected. Therefore, we can assume, if desired, in the statement of Corollary \ref{cor:pa-to-lox-multi} that the witness $W$ is connected.
    \end{remark}

    Theorem \ref{thm:paimplieslox} can also be applied to the disk graph.
    \begin{cor}\label{cor:pa-to-lox-disk}
        For $g\geq 2$, given $h\in \hg$ if some pure power of $h$ restricts to a pseudo-Anosov on a witness for $\dg$, then $h$ acts loxodromically on $\dg$.
    \end{cor}
    
    \begin{proof}    
        For all $g\geq 2$, $\hg$ is non-trivial and non-cyclic and $\dg$ is infinite. In these cases, since $\hg$ acts by isometries on $\hg$ and has a witness-based distance formula by Theorem \ref{thm:disk-dist-formula}, by our assumption on $h$ and Theorem \ref{thm:paimplieslox}, $h$ must act loxodromically on $\dg$.
    \end{proof}

\section{Loxodromic implies pseudo-Anosov on a witness}\label{sec:lox-implies-pA}

In this section, we prove that loxodromic elements are pseudo-Anosovs on witnesses in the cases of admissible multiarc and curve graphs and disk graphs. The strategies for each case are different. For admissible multiarc and curve graphs, we make use of a particular combinatorial complex constructed by Kopreski (based on a construction of Vokes \cite{kopreski2023multiarc,vokes}). The proof for disk graphs relies more on the topological features of handlebodies, as the constructions of Vokes and Kopreski do not apply to this case.

\subsection{Admissible multiarc and curve graphs}\label{subsec:multiarccurve-proof}

For this section, assume that $\g$ is an admissible multiarc and curve graph for a surface $S$. We record here several definitions and facts from Kopreski \cite{kopreski2023multiarc} that we will use in our proof.

\begin{defn}[{\cite[Definition 2.1]{kopreski2023multiarc}}]\label{def:marking}
    A \emph{marking} $\mu=\{(a_i,t_i)\}$ for a surface $S$ consists of a set of \emph{base curves}, $\operatorname{base}(\mu)=\{a_i\}$, and a set of \emph{transversals}, $\operatorname{trans}(\mu)=\{t_i\}$. The set of base curves $\{a_1,\ldots, a_n\}$ is required to be a non-empty, essential, simple multicurve on $S$. Each transversal $t_i$ is a (possibly empty) diameter-1 subset of the annular curve graph associated the the annulus with core curve $a_i$.
\end{defn}

For the next definition, note that the complexity of a surface $S$ is defined as $\xi(S)=3g+b-3$, where $g$ is the genus and $b$ is the number of boundary components.

\begin{defn}[{\cite[Definition 2.2, Remark 2.3]{kopreski2023multiarc}}]\label{def:clean}
    A marking $\mu=\{(a_i,t_i)\}$ is \emph{locally complete} if whenever $t_i\neq \varnothing$, then the component of $\operatorname{base}(\mu)-\{a_i\}$ that contains $a_i$ has complexity $\xi =1$. We say that $\mu$ is a \emph{locally clean (or just clean)} if it is also the case that whenever $t_i\neq \varnothing$ there is a curve $b_i$ satisfying the following:
    \begin{itemize}
        \item[(i)] the subsurface $F$ filled by $a_i\cup b_i$ has complexity $\xi=1$,
        \item[(ii)] $a_i$ and $b_i$ are adjacent in the curve graph $\mathcal{C}F$,
        \item[(iii)] $b_i\cap a_j=\varnothing$ for $i\neq j$, and
        \item[(iv)] $t_i=\pi_{a_i}(b_i)$, (that is $t_i$ is the projection of $b_i$ to the curve complex of the annulus with core curve $a_i$).
    \end{itemize}
\end{defn}

\begin{defn}[{\cite[Definition 2.7]{kopreski2023multiarc}}]\label{def:meets}
    Let $\mu=\{(a_i,t_i)\}$ be a clean marking on a surface $S$, and let $W\subseteq S$ be an essential (possibly disconnected) subsurface. We say that $\mu$ \emph{meets} $W$ if either
    \begin{enumerate}
        \item[(i)] $W$ intersects $\{a_i\}$, or 
        \item[(ii)] $W$ has an annular component whose core is some $a_i$ with $t_i\neq \varnothing$.
    \end{enumerate}
\end{defn}

A subgraph of the following graph will be needed for our proof of Theorem \ref{thm:loximpliespA-multiarccurve}.

\begin{defn}[{\cite[Definition 2.14]{kopreski2023multiarc}}]\label{def:L-graph}
    Let $\g$ be an admissible multiarc and curve graph, and let $\wit$ be the set of all connected witnesses of $\g$. Define the graph $\lwit$ to be the simplicial graph whose vertices are all clean markings that meet every surface in $\wit$, and for which there is an edge between two markings $\mu,\nu$ exactly when $\mu$ is obtained from $\nu$ by either
    \begin{enumerate}
        \item[(i)] adding or removing a component $(a,t)$,
        \item[(ii)] adding or removing a transversal, or
        \item[(iii)] an elementary move.\footnote{The exact definition of an elementary move will not be relevant to this paper. See \cite[Definition 2.13]{kopreski2023multiarc} for the definition.}
    \end{enumerate}
\end{defn}

The result below follows from Section 2.1.1 and Section 3 of \cite{kopreski2023multiarc}.
\begin{thm}[{\cite{kopreski2023multiarc}}]\label{thm:L-qi}
    Given a connected, admissible multiarc and curve graph $\g$, there is a coarsely $\pmcg$-equivariant quasi-isometry $\phi:\g \to \lwit$.
\end{thm}

We also note the following, which key in our proof of Theorem \ref{thm:loximpliespA-multiarccurve}, and is akin the the requirement on the vertices of the graph $\antiwit(S)$ as defined by Vokes \cite[Definition 3.1]{vokes}.

\begin{lem}\label{lem:complement-witnesses}
    Let $\{a_i\}$ be a multicurve on $S$ such that that no component of $S-\{a_i\}$ is a witness. Then the marking $\mu$ with base curves $\{a_i\}$ and transversals $t_i=\varnothing$ for each $i$ is a vertex of $\lwit$.
\end{lem}

\begin{proof}
    The marking $\mu$ vacuously satisfies the definition of a clean marking as being clean is a condition only on any non-empty transversals. Furthermore, due to all the transversals being empty, to show that $\mu$ meets every $W\in\wit$, we must show that $\mu$ satisfies (i) from Definition \ref{def:meets} for each $W\in\wit$.
    
    For contradiction, assume that some $W\in\wit$ did not meet $\mu$. Then by Definition \ref{def:meets}, $W$ must have empty intersection with every $a_i$. This would imply that $W$ must be a subset of $S-\{a_i\}$, and since $W$ is assumed to be connected, it must be contained in a single component, $V$, of $S-\{a_i\}$. But if $W$ is a witness and $W\subseteq V$, then $V$ must also be a witness, contradicting the fact that no component of $S-\{a_i\}$ is a witness. We must then conclude that $\mu$ meet every $W\in\wit$ and thus, $\mu\in\lwit^{(0)}$.
\end{proof}

We now have the required tools to prove the main theorem of this subsection.

\begin{thm}\label{thm:loximpliespA-multiarccurve}
    Let $G\leq \mcg$ act by isometries on a connected, admissible multiarc and curve graph $\g$. Further, assume $h\in G$ acts loxodromically on $\g$. Then either $h$ itself or some pure power of $h$ restricts to a pseudo-Anosov on a connected witness for $\g$.
\end{thm}

Before we prove Theorem \ref{thm:loximpliespA-multiarccurve}, we make one observation regarding connectedness of the witness in the theorem statement. The graph we denote in this paper by $\lwit$ is denoted in \cite{kopreski2023multiarc} as $\mathcal{L}_{\hat{\mathscr{X}}}$. Kopreski also defines the graph $\mathcal{L}_{\mathscr{X}}$ as an analogous graph where witnesses are allowed to be disconnected. Both graphs admit $\pmcg$-equivariant quasi-isometries to the associated admissible multiarc and curve graph $\g$, (see the remark immediately preceding Section 2.1.2 in \cite{kopreski2023multiarc}), with disconnected witnesses being required to define the hierarchy structure on $\g$. Since the proof of Theorem \ref{thm:loximpliespA-multiarccurve} does not require the hierarchy structure on $\g$ and simply requires a $\pmcg$-equivariant quasi-isometry, this allows us to require the witness in the statement of Theorem \ref{thm:loximpliespA-multiarccurve} to be connected.

\begin{proof}[Proof of Theorem \ref{thm:loximpliespA-multiarccurve}]
    First, we note that $h$ cannot be periodic because this would contradict the fact that $h$ acts loxodromically on $\g$.

    If $h$ is pseudo-Anosov, $h$ is pseudo-Anosov on a witness since the whole surface $S$ is always a witness for any $\g$.

    It then remains to consider the case that $h$ is reducible. Let $\crs$ be the canonical reduction system, $f$ a pure power of $h$, and $\{X_j\}$ the set of components of  $S - n(\crs)$. Note that since $h$ acts loxodromically on $\g$, so must $f$.
    
    Let $\bar{\phi}:\lwit\to\g$ denote a quasi-inverse of the quasi-isometry $\phi:\g \to \lwit$ defined in Theorem \ref{thm:L-qi}; because $\phi$ is a coarsely $\pmcg$-equivariant quasi-isometry, so to is $\bar{\phi}$.
    For contradiction, assume that no component of $\{X_j\}$ is a witness. By Lemma \ref{lem:complement-witnesses}, this implies that the marking $\mu_{\crs}$ with base curves $\crs$ and empty transversals is a vertex of $\lwit$. Since $\crs$ is the canonical reduction system of $f$ and is therefore invariant under the action of $f$, $\mu_{\crs}$ is a fixed point in the action of $f$ on $\lwit$. By Lemma \ref{lem:finitediam}, we can then conclude that the diameter of the set of $f$ translates of $\bar{\phi}(\mu_{\crs}) \in \g$ is finite. This contradicts the fact that $f$ acts loxodromically on $\g$. Therefore, it must be the case that at least one component of $\{ X_j\}$ is a witness.

    We now argue that one of the witness components contained in $\{X_j\}$ must be a pseudo-Anosov component. To see why, we choose a pants decomposition for each of the witness components in $\{X_j\}$, and define a multicurve $\pants$ to be the union of those pants decompositions. Notice that $S-n(\crs \cup \pants)$ consists of pairs of pants and the components of $\{X_j\}$ that are not witnesses. Because we do not consider pairs of pants to be witnesses, the marking $\mu_{\crs\cup\pants}$ with base curves $\crs \cup \pants$ and empty transversals is a vertex of $\lwit$ by Lemma \ref{lem:complement-witnesses}. Additionally, since $\crs$ is fixed by $f$, if it were the case that none of the witnesses in $\{X_j\}$ were pseudo-Anosov components, then $\pants$, and hence $\crs \cup \pants$, would also be fixed by $f$. Again, applying Lemma \ref{lem:finitediam} to $\mu_{\crs\cup \pants}$, as we did with $\mu_{\crs}$ before, gives us a contradiction to the assumption that $f$ acts loxodromically on $\g$. Hence, it must be the case that one of the witnesses in $\{X_j\}$ is a pseudo-Anosov component.
\end{proof}

\subsection{Disk graphs}\label{subsec:diskgraphproof}

For the entirety of this section, let $V_g$ denote a genus $g\geq 2$ handlebody and let $S=\bddy V_g$, which is homeomorphic to an orientable surface of genus $g$.

Most of the work to prove that an element of the handlebody group $\hg$ acting loxodromically on $\dg$ must restrict a pseudo-Anosov on a witness comes from the case when $\bddy f$ is a reducible mapping class. We further split this case into two subcases, accounted for in Lemmas \ref{lem:compressible} and \ref{lem:incompressible} below.

\begin{lem}\label{lem:compressible}
    Suppose $f\in\hg$ acts loxodromically on the disk graph and is reducible with canonical reduction system $\crs$. If any component $W$ of $S-n(\redsys)$ is compressible in $V_g$, then $f$, and therefore any pure power of $f$, must restrict to a pseudo-Anosov on $W$ and $W$ must be a witness.
\end{lem}

\begin{proof}
    
    Let $W$ be a compressible component of $S-n(\redsys)$. Then there is some disk $D$ contained in $W$. Let us first see why $W$ must actually be invariant under the action of $f$. Suppose instead that for some power of $f$, say $f^n$, that $f^n(W) \cap W = \varnothing$. Then $f^n(D)$ is disjoint from $D$, and in fact there must be infinitely many powers of $f^n$, say $\{f^{n_i}\}_{i=1}^{\infty}$ such that each $f^{n_i}(D)$ is disjoint from $D$. This would contradict the fact that $f$ acts loxodromically on $\dg$ because these disks would all be at distance at most two from each other in $\dg$. So $W$ must be invariant under  the action of $f$. It must also be the case that $W$ a pseudo-Anosov component because otherwise $D$ would provide a fixed point in the action of $f$ on $\dg$.
    
    We now show that $W$ must be a witness. Suppose for contradiction that $W$ was not a witness. This would mean there is some disk $E$ contained in $S-W$. Since $W$ is invariant under the action of $f$, it must be the case that for any $k$, $f^k(D) \subseteq W$. It then follows that $f^k(D)$ is disjoint from $E$ for each $k$, meaning that all powers $f^k(D)$ are always distance at most two from each other. Thus, $f$ must not act loxodromically on $\dg$ -- a contradiction. Therefore we have shown that $W$ is both a pseudo-Anosov component for $f$ and a witness for $\dg$.
\end{proof}

\begin{remark}
    We note that the preceding proof also shows that for $f$ as in the statement of Lemma \ref{lem:compressible}, only one component of $S-n(\crs)$ can be compressible if $f$ is loxodromic.
\end{remark}

Some additional definitions and facts will be required for the next proof. The first fact concerns which Dehn twists are contained in $\hg$.
\begin{thm}[{{\cite[Theorem 1.11]{oertel-handlebody}, \cite[Theorem 1]{McCullough-twist}}}]\label{thm:twists}
    Suppose $h\in \hg$ is a composition of
    non-trivial
    Dehn twists about simple, closed curves $c_1, \ldots, c_n$. Then for each $i$, either $c_i$ bounds a disk, or there is a $j\neq i$ such that $c_i$ and $c_j$ cobound a properly embedded, incompressible annulus in $V_g$.
\end{thm}
When we have a composition of twists $T_{c_i}T_{c_j}^{-1}$ where $c_i$ and $c_j$ cobound an annulus as in the preceding theorem, we will refer to this composition as an \emph{annulus twist}.

When considering two curves $a$ and $b$ in $S$, we may sometimes assume that these curves are \emph{tight}, meaning that they realize their geometric intersection number.

We must also introduce some definitions concerning various properties and types of three-manifolds. Let $F$ be a compact surface; following \cite{oertel-handlebody}, $F$ is allowed to be disconnected. We will consider $I$-bundles (interval bundles) $p:H\to F$ where the base space $F$ is a surface. We define the \emph{vertical boundary} $\vbddy H$ of $H$ as $p^{-1}(\bddy F)$. The \emph{horizontal boundary} $\hbddy H$ of $H$ is the closure of $\bddy H - \vbddy H$.
Note that these are referred to as the exterior and interior boundaries of an $I$-bundle in \cite{oertel-handlebody}. As all $I$-bundles we consider will be submanifolds of orientable handlebodies, the $I$-bundles themselves will be orientable.

For an $I$-bundle $p:H\to F$ as in the previous paragraph, we will say that a homeomorphism $f:H\to H$ is a \emph{lift} of an homeomorphism $g:F\to F$ when $p\circ f = g\circ f$.

Oertel proves the following facts in relation to homeomorphisms of handlebodies which will be useful in proving Lemma \ref{lem:incompressible}.

\begin{lem}[{{\cite[Lemma 2.13 and proof]{oertel-handlebody}}}]\label{thm:oertel-pa}
    Let $f:H\to H$ be a homeomorphism of a handlebody $H$ such that $\bddy f$ is reducible on $\bddy H$. Let $\crs$ be the canonical reduction system and suppose all components of $R=H-n(\crs)$ are incompressible. If $\bddy f$ is pseudo-Anosov on some component of $R_0$ of $R$, then there is a submanifold $J\hookrightarrow H$, the total space of an $I$-bundle, with $\hbddy J \subseteq \bddy H$, and with (up to isotopy) the restriction of $f$ to $J$ the lift of a pseudo-Anosov homeomorphism of a surface. 
    Moreover, $\hbddy J$ is the union of the pseudo-Anosov components of $R$.
\end{lem}

\begin{prop}[{{\cite[Proposition 2.15 and proof]{oertel-handlebody}}}]\label{prop:oertel-annular-red}
    Let $f:H\to H$ be a homeomorphism of a connected handlebody $H$. Let $\crs$ be the canonical minimal reduction system for the induced homeomorphism $\bddy f$ of $\bddy H$. Suppose $R=H-n(\crs)$ is incompressible and some component of $R$ is a pseudo-Anosov component. Then there exists a submanifold $J\hookrightarrow H$, the total space of an $I$-bundle satisfying the properties in Lemma \ref{thm:oertel-pa}, and either
    \begin{enumerate}
        \item[(1)] $H=J$, or
        \item[(2)] $J$ is a proper submanifold of $H$ and $\vbddy J$ contains an $f$-invariant annulus $A$ that is incompressible and non-boundary parallel in $H$, with the components of $\bddy A$ isotopic in $\bddy H$ to curves in $\crs$.\label{prop:oertel-annular-red-annulus}
    \end{enumerate}
\end{prop}

Finally, we will need the techniques relating to boundary compressions and corresponding surgeries; we follow the conventions of Masur and Schleimer \cite[Section 8]{MasSchleim}. First, a triple $(B,\alpha, \beta)$ is called a \emph{bigon} if $B$ is a disk and $\alpha$ and $\beta$ are arcs in $\bddy B$ such that $\alpha\cup \beta = \bddy B$ and $\alpha \cap \beta = \bddy\alpha = \bddy \beta$. Given a three-manifold $M$ and surface $F$ that is either contained in $\bddy M$ or is properly embedded in $M$, an embedded bigon $(B,\alpha,\beta)\subseteq (M,F, \bddy M)$ is called a \emph{surgery bigon for $F$} if
\begin{itemize}
    \item $B\cap \bddy F = \bddy\alpha = \bddy \beta$,
    \item $B\cap F = \alpha$, and
    \item $B\cap \bddy M =\bddy B$ if $F\subseteq \bddy M$, or $B\cap \bddy M = \beta$ if $F$ is properly embedded in $M$.
\end{itemize}
The presence of a surgery bigon for $F$ allows us to surger $F$ along $B$ in the following manner. Construct a closed regular neighborhood $N$ about $B$. Remove the rectangle $N\cap F$ from $F$. Then form the surgery $F_B$ by gluing on the two bigon components of $\operatorname{fr}(N)-F$ and taking the closure. In the case that $F\subset \bddy M$, we also isotope the interior of $F_B$ into the interior of $M$ to guarantee that $F_B$ is properly embedded in $M$.

We now have all the necessary tools to prove the following lemma.

\begin{lem}\label{lem:incompressible}
    Suppose $f\in\hg$ acts loxodromically on the disk graph and is reducible with canonical reduction system $\crs$. Suppose that every component of $S-n(\redsys)$ is incompressible. Then there is some subsurface $W\subseteq S$ that is a witness for $\dg$ and on which some pure power of $f$ restricts to a pseudo-Anosov.
\end{lem}
\begin{proof}
    Suppose $\crs$ is the canonical reduction system of $f$ and that $h$ is a pure power of $f$. Consider Theorem \ref{thm:twists}, which states that the only non-trivial Dehn twists in the handlebody groups are either twists about disks or annulus twists. 
    Because $f$ acts loxodromically on $\dg$, no curve in $\crs$ can be a disk, else there would be a fixed point under the action of $f$ on $\dg$. Therefore, if $c\in \crs$, then either $h$ acts as a trivial Dehn twist along $c$, or twisting along $c$ is part of an annulus twist, along with a twist on some other curve in $\crs$. 
    
    We let $\annuli$ be a collection of properly embedded annuli bounded by the curves in $\crs$ that correspond to annulus twists. We further define $\{H_i\}_{i=1}^n$ to be the components of $V_g-n(\annuli)$.

    \begin{claim}\label{claim:H-compressible}
        There is an $H\in \{H_i\}_{i=1}^n$ such that $\bddy H$ is compressible in $V_g$.
    \end{claim}

    \begin{claimproof}
        Let $E$ be an essential disk in $V_g$ with boundary curve that is tight with $\crs$. If $E$ lies entirely in some $H\in \{H_i\}_{i=1}^n$, then we are done. So suppose instead that $E$ does not lie entirely within any $H\in \{H_i\}_{i=1}^n$.

        Consider the set $\annuli \cap E$. We note that this set cannot contain any circles because the annuli coming from annulus twists are incompressible by Theorem \ref{thm:twists}. 
        Thus, $\annuli \cap E$ must consist of a collection of disjoint arcs. Consider now an outermost arc $\alpha$ of $\annuli \cap E$ in $E$, and let $A$ be the annulus in $\annuli$ containing $\alpha$. We have two cases to consider.

        \textit{Case 1:} $\alpha$ is inessential in $A$. In this case, $\alpha$ has both endpoints on a single component of $\bddy A$. The set $A - \alpha$ will consist of two components, one of which will be a bigon $B$ with sides $\alpha$ and $\beta\subseteq \bddy A$.
        Notice then that $(B, \alpha, \beta) \subseteq (V_g, E, \bddy V_g)$ forms a surgery bigon for $E$. 
        Surger $E$ along $B$ to obtain $E_B$, which will have two components. The components of $E_B$ must be essential disks in $V_g$ because otherwise $E$ would not have been tight with $\crs$; in particular, we could have performed an ambient isotopy of $E$ along $E_B$, removing the arc of intersection $\alpha$. Moreover, since we chose $\alpha$ to be outermost, one of the components of of $E_B$ must be contained entirely within some $H\in \{H_i\}_{i=1}^n$; call this component $E_B'$. Since $E_B'\subseteq H$ is essential, $\bddy H$ must be compressible in $V_g$.

        \textit{Case 2:} $\alpha$ is essential in $A$. Since we chose $\alpha$ to be outermost, there is an arc $\beta\subseteq \bddy E$ sharing endpoints with $\alpha$ that must lie within some component $H\in \{H_i\}_{i=1}^n$. Then the portion of $E$ bounded by $\alpha$ and $\beta$ forms a surgery bigon $(B,\alpha,\beta)\subseteq(V_g, A, \partial V_g)$ for $A$. The disk $A_B$ resulting from surgering $A$ along $B$ can be isotoped to be contained in $H$ and must again be essential. Therefore, $\partial H$ is compressible in $V_g$.
    \end{claimproof}

    Let $\{X_j\}_{j=1}^m$ be the components of $H \cap (\bddy V_g-n(\crs))$.

    \begin{claim}\label{claim:pA-component}
        One of the components $X_k\in \{X_j\}_{j=1}^m$ in $H$ must be a pseudo-Anosov component.
    \end{claim}

    \begin{claimproof}
        Recall $H$ is a component of $V_g - n(\annuli)$, and $h$ acts as a trivial Dehn twist along any curves in $\crs - \partial \annuli$. This means that if every component $\{X_j\}$ was an identity component, then $h$ would be the identity on $H$. But by Claim \ref{claim:H-compressible}, there must be some disk $E\subseteq H$. Therefore, $E$ would be a fixed point in the action of $h$ on $\dg$ -- a contradiction since $f$, and therefore $h$, act loxodromically on $\dg$. Hence, there must be some component $X_k$ that is a pseudo-Anosov component.
    \end{claimproof}

    At this point, we have that:
    \begin{itemize}
        \item $H$ contains a pseudo-Anosov component $X_k$ by Claim \ref{claim:pA-component}, and 
        \item every component of $\partial H - n(\crs)$ is incompressible (since we assumed that every component of $\partial V_g - n(\crs)$ is incompressible).
    \end{itemize}
    Thus, Proposition \ref{prop:oertel-annular-red} applies to $h|_{H}:H\to H$. Because Proposition \ref{prop:oertel-annular-red-annulus} (2) does not apply by construction, we can conclude that $H$ itself has the structure of an $I$-bundle such that $H$ satisfies the assumptions of $J$ as described in Theorem \ref{thm:oertel-pa}. In particular, $h|_H$ is the lift of a pseudo-Anosov over some (potentially non-orientable) surface, and moreover $\hbddy H$ is the union of the pseudo-Anosov components in $\bddy H$. Additionally, as $H$ is connected by construction, each component of $\hbddy H$ is itself a pseudo-Anosov component (there are either one or two components of $\hbddy H$ depending on if the base space is a non-orientable or orientable surface).
    We also note that $\vbddy H$ is the union of some annuli, some of which may be in $\annuli$; the remaining annuli will be boundary parallel in $V_g$.

    \begin{claim}\label{claim:pa-witness}
        The pseudo-Anosov components contained in $H$ are witnesses.
    \end{claim}

    \begin{claimproof}
        Let $E$ be an essential disk in $V_g$. We note first that $E$ must intersect $H$; this follows because $H$ is compressible and is invariant under $h$ since $h$ is a pure power, so an argument similar to that given in the last paragraph of the proof of Lemma \ref{lem:compressible} shows that $H$ must be a witness.
        
        Now apply an ambient isotopy so that $\bddy E$ is tight with $\crs$. Notice that $\bddy E$ cannot be contained entirely within any single component of $\vbddy H \cap \bddy V_g$ nor $\hbddy H \cap \bddy V_g$ because each of those components is either a component of $\bddy V_g - n(\crs)$, which by assumption are incompressible in $V_g$, or is an annular neighborhood in $\bddy V_g$ of a non-annulus curve in $\crs$ which as stated earlier, cannot be disks. This means that $E$ has to cross at least one component of $\vbddy H$ and one component of $\hbddy H$.

        The intersection $E \cap \vbddy H$ must be a collection of disjoint arcs since each component of $\vbddy H$ is an incompressible annulus. Consider an arc $\alpha$ that is outermost in $E$ and let $A$ be the component of $\vbddy H$ containing $\alpha$. If $\alpha$ is essential, then we are done because this means $E$ must cross all components of $\hbddy H$.

        To see why $\alpha$ must be essential, consider the two possibilities for $A$. First, if $A$ was a boundary parallel annulus in $V_g$ and $\alpha$ were inessential, then $\alpha$ could have been removed via an ambient isotopy of $E$, contradicting the fact that $E$ was tight with $\crs$. If $A$ was not a boundary parallel annulus in $V_g$, then $A \in \annuli$, and as in the proof of Claim \ref{claim:H-compressible}, we could perform a surgery to obtain a disk $E_B$ contained within $H$; but this disk $E_B$ would have to be contained within a single component of $\hbddy H$, contradicting the fact that each component of $\hbddy H$ is incompressible. Therefore, $\alpha$ must be essential in $A$.
        
        Hence, we can conclude that $E$ must cross each component of $\hbddy H$ (of which there are at most two), meaning that it crosses the pseudo-Anosov components contained in $\bddy H$. Since $E$ was an arbitrary essential disk in $V_g$, we conclude that every disk must cross the pseudo-Anosov components in $\bddy H$, and therefore the pseudo-Anosov components must be witnesses.
    \end{claimproof}

    If follows directly from Claim \ref{claim:pa-witness} that $h$ restricts to a pseudo-Anosov on a witness.
\end{proof}

Putting together the previous two lemmas we can conclude the following.

\begin{thm}\label{thm:lox-implies-pa-handlebody}
    Suppose $f\in\hg$ acts loxodromically on $\dg$. Then $f$ is either is either pseudo-Anosov or reducible and some pure power of $f$ must restrict to a pseudo-Anosov on a witness for $\dg$.
\end{thm}

\begin{proof}
    We may consider two cases:  $\bddy f$ is reducible or pseudo-Anosov as a mapping class.
    We note that it cannot be the case that $\bddy f$ is periodic: if $\bddy f$ is periodic, then some power of $\bddy f$ is the identity, and so the action on $\dg$ cannot have positive translation length, i.e. $f$ cannot have been loxodromic.
    
    The pseudo-Anosov case is straightforward. First we note that there do exist elements of the handlebody group such that their restriction to the boundary is a pseudo-Anosov (see \cite[Section 2]{zbMATH05052470}). Additionally, if $\bddy f$ is pseudo-Anosov, then $\bddy f$ acts loxodromically on the curve graph (see Theorem \ref{thm:masur-minsky-original}) and therefore also act loxodromically on the disk graph. Finally, $S=\bddy V_g$ is a witness for $\dg$ since every disk is contained in $S$. Thus, $\bddy f$ can be a pseudo-Anosov, and in that case it is pseudo-Anosov on a witness.

    Finally, when $\bddy f$ is reducible with canonical reduction system $\crs$, then either some component of $S-n(\crs)$ is compressible or all components are incompressible. Lemmas \ref{lem:compressible} and \ref{lem:incompressible} cover each of these two cases and guarantee that in either scenario, some pure power of $f$ restricts to a pseudo-Anosov on a witness for $\dg$.
\end{proof}

\bibliographystyle{amsalpha}
\bibliography{bib}

@article{MasSchleim,
    title="The Geometry of the Disk Complex",
    author="Howard Masur and Saul Schleimer",
    journal="J. Amer. Math. Soc.",
    volume="26",
    number="1",
    pages="1--62",
    year="2013"
}

@Book{BH,
 author    = {Martin R. Bridson and André Häfliger},
 title     = {Metric Spaces of Non-Positive Curvature},
 publisher = {Springer},
 year      =  {1999},
 address   = {Berlin},
 series    = {Grundlehren der mathematischen Wissenschaften},
 volume    = {319}
}

@misc{russell2019nonrelative,
      title={The (non)-relative hyperbolicity of the separating curve graph}, 
      author={Jacob Russell and Kate M. Vokes},
      year={2019},
      eprint={1910.01051},
      archivePrefix={arXiv},
      primaryClass={math.GT}
}

@Article{chesser22,
 Author = {Chesser, Marissa},
 Title = {Stable subgroups of the genus 2 handlebody group},
 FJournal = {Algebraic \& Geometric Topology},
 Journal = {Algebr. Geom. Topol.},
 ISSN = {1472-2747},
 Volume = {22},
 Number = {2},
 Pages = {919--971},
 Year = {2022},
 Language = {English},
 DOI = {10.2140/agt.2022.22.919},
 zbMATH = {7570609}
}

@misc{kopreski2023multiarc,
      title={Multiarc and curve graphs are hierarchically hyperbolic}, 
      author={Michael Kopreski},
      year={2023},
      eprint={2311.04356},
      archivePrefix={arXiv},
      primaryClass={math.GT}
}

@Article{vokes,
 Author = {Vokes, Kate M.},
 Title = {Hierarchical hyperbolicity of graphs of multicurves},
 FJournal = {Algebraic \& Geometric Topology},
 Journal = {Algebr. Geom. Topol.},
 ISSN = {1472-2747},
 Volume = {22},
 Number = {1},
 Pages = {113--151},
 Year = {2022},
 Language = {English},
 DOI = {10.2140/agt.2022.22.113},
 zbMATH = {7518302}
}

@Article{oertel-handlebody,
 Author = {Oertel, Ulrich},
 Title = {Automorphisms of three-dimensional handlebodies},
 FJournal = {Topology},
 Journal = {Topology},
 ISSN = {0040-9383},
 Volume = {41},
 Number = {2},
 Pages = {363--410},
 Year = {2002},
 Language = {English},
 DOI = {10.1016/S0040-9383(00)00041-0},
 zbMATH = {1717412},
 Zbl = {0991.57017}
}

@Article{McCullough-twist,
 Author = {McCullough, Darryl},
 Title = {Homeomorphisms which are {Dehn} twists on the boundary},
 FJournal = {Algebraic \& Geometric Topology},
 Journal = {Algebr. Geom. Topol.},
 ISSN = {1472-2747},
 Volume = {6},
 Pages = {1331--1340},
 Year = {2006},
 Language = {English},
 DOI = {10.2140/agt.2006.6.1331},
 zbMATH = {5118562},
 Zbl = {1135.57011}
}

@Article{masur-minsky,
 Author = {Masur, H. A. and Minsky, Y. N.},
 Title = {Geometry of the complex of curves. {II}: {Hierarchical} structure},
 FJournal = {Geometric and Functional Analysis. GAFA},
 Journal = {Geom. Funct. Anal.},
 ISSN = {1016-443X},
 Volume = {10},
 Number = {4},
 Pages = {902--974},
 Year = {2000},
 Language = {English},
 DOI = {10.1007/PL00001643},
 zbMATH = {1545126},
 Zbl = {0972.32011}
}

@article{masur-minsky-I,
 author = {Masur, Howard A. and Minsky, Yair N.},
 title = {Geometry of the complex of curves. {I}: {Hyperbolicity}},
 fjournal = {Inventiones Mathematicae},
 journal = {Invent. Math.},
 issn = {0020-9910},
 volume = {138},
 number = {1},
 pages = {103--149},
 year = {1999},
 language = {English},
 doi = {10.1007/s002220050343},
 zbMATH = {1355494},
 Zbl = {0941.32012}
}

@Article{HHSII,
 Author = {Behrstock, Jason and Hagen, Mark and Sisto, Alessandro},
 Title = {Hierarchically hyperbolic spaces. {II}: {Combination} theorems and the distance formula},
 FJournal = {Pacific Journal of Mathematics},
 Journal = {Pac. J. Math.},
 ISSN = {1945-5844},
 Volume = {299},
 Number = {2},
 Pages = {257--338},
 Year = {2019},
 Language = {English},
 DOI = {10.2140/pjm.2019.299.257},
 zbMATH = {7062864},
 Zbl = {1515.20208}
}

@article{HHSI,
 author = {Behrstock, Jason and Hagen, Mark F. and Sisto, Alessandro},
 title = {Hierarchically hyperbolic spaces. {I}: {Curve} complexes for cubical groups},
 fjournal = {Geometry \& Topology},
 journal = {Geom. Topol.},
 issn = {1465-3060},
 volume = {21},
 number = {3},
 pages = {1731--1804},
 year = {2017},
 language = {English},
 doi = {10.2140/gt.2017.21.1731},
 zbMATH = {6726511},
 Zbl = {1439.20043}
}

@misc{notescomplexofcurves,
  author = {Saul Schleimer},
  title = {Notes on the complex of curves},
  note = {Last accessed 11 November 2025},
  url = {https://sschleimer.warwick.ac.uk/Maths/notes2.pdf}
}

@article{zbMATH05052470,
 author = {Carvalho, Leonardo Navarro},
 title = {Tightness and efficiency of irreducible automorphisms of handlebodies},
 fjournal = {Geometry \& Topology},
 journal = {Geom. Topol.},
 issn = {1465-3060},
 volume = {10},
 pages = {57--95},
 year = {2006},
 language = {English},
 doi = {10.2140/gt.2006.10.57},
 url = {https://eudml.org/doc/127722},
 zbMATH = {5052470},
 Zbl = {1116.57017}
}

\end{document}